\documentclass{amsart}
\usepackage[utf8]{inputenc}

\usepackage{amsthm,amsmath,xcolor}
\usepackage{graphicx}
\usepackage{caption, subcaption, comment,rotating, lscape, url}

\newtheorem{theorem}{Theorem}[section]
\newtheorem{definition}[theorem]{Definition}
\newtheorem{example}[theorem]{Example}
\newtheorem{remark}[theorem]{Remark}
\newtheorem{lemma}[theorem]{Lemma}
\newtheorem{conj}[theorem]{Conjecture}

\title{Computational progress on the unfair 0-1 polynomial Conjecture}

\author{Kevin G. Hare}
\email{kghare@uwaterloo.ca}
\address{Department of Pure Mathematics, University of Waterloo, Waterloo, Ontario, Canada N2L 3G1}
\thanks{Research of K.G. Hare was supported in part by NSERC Grant 2019-03930}

\begin{document}

\begin{abstract}
Let $c(x)$ be a monic integer polynomial with coefficients $0$ or $1$.
Write $c(x) = a(x) b(x)$ where $a(x)$ and $b(x)$ are monic polynomials with non-negative real (not necessarily integer) coefficients.
The unfair 0--1 polynomial conjecture states that $a(x)$ and $b(x)$ are necessarily integer polynomials with coefficients $0$ or $1$.
Let $a(x)$ be a candidate factor of a (currently unknown) 0--1 polynomial.  We will assume that we know if a coefficient is $0$, $1$ or strictly between $0$ and $1$, but that we do not know the precise value of non-integer coefficients.  Given this candidate $a(x)$, this paper gives an algorithm to either find a $b(x)$ and $c(x)$ with $a(x) b(x) = c(x)$ such that $b(x)$ has non-negative real coefficients and $c(x)$ has coefficients $0$ or $1$, or (often) shows that no such $c(x)$ and $b(x)$  exist.
Using this algorithm, we consider all candidate factors with degree less than or equal to 15.
With the exception of 975 candidate factors (out of a possible 7141686 cases), this algorithm shows that there do not exist $b(x)$ with non-negative real coefficients and $c(x)$ with coefficients $0$ or $1$ such that $a(x) b(x) = c(x)$.
\end{abstract}

\maketitle

\section{Introduction}
\label{sec:intro}

Let $X$ and $Y$ be independent discrete random variables with finite support.  
Then $Z = X + Y$ is also a discrete random variable with finite support.
It is conjectured that if $Z$ is uniform on its support, then $X$ and $Y$ must also be uniform on their support.
To the author's knowledge, this was first asked by G. Letac in 1969.

This conjecture can be translated to a conjecture about polynomial multiplication in the following way.
Let $c(x)$ be a monic integer polynomial with coefficients $0$ or $1$.
Factor $c(x)$ as $c(x) = a(x) b(x)$ where $a(x)$ and $b(x)$ monic polynomials with non-negative real coefficients.
Is it true that $a(x)$ and $b(x)$ are necessarily integer polynomials with coefficients $0$ or $1$?
This is clearly not true if we relax the restriction that $a(x)$ and $b(x)$ are monic.  
Simply take $a(x) = 2$ and $b(x) = \frac{1}{2} c(x)$ as an example.
Further, this is not true if we relax the restriction on the coefficients being non-negative real.
Take for example $x^3+1 = (x+1)(x^2-x+1)$ or $x^2+1 = (x+i)(x-i)$.
We will say that a 0--1 polynomial $c(x)$ is {\em fair} if all factorizations $c(x) = a(x) b(x)$ with 
    $a(x)$, $b(x)$ monic polynomials with real non-negative coefficients have the property that  $a(x)$ and $b(x)$ are 0--1 polynomials.
We say that a 0--1 polynomial $c(x)$ is {\em unfair} if there exists a factorization $c(x) = a(x) b(x)$ with $a(x), b(x)$ 
    both monic with real non-negative coefficients, and at least one of $a(x)$ or $b(x)$ has a non-integer coefficient.
The unfair 0--1 polynomial conjecture is that there does not exist an unfair polynomial.
See \cite{Ghidelli} and the expansive references therein for more details.

The goal of this paper is to provide a classification for those $a(x)$ of degree at most 15 which are not factors of an unfair polynomial.

Let $a(x) = a_0 + a_1 x + \dots + a_k x^k$ be a potential factor of an 
    unfair polynomial.
We further assume that the shape of $a(x) = a_0 + a_1 x + \dots + a_k x^k$ is given.
By this, we mean that it is known if $a_i$ is $0, 1$ or strictly between $0$ and $1$.
We wish to determine if $a(x)$ is actually a factor of an unfair 
    polynomial.
That is, we wish to determine if there exists a $b(x) = b_0 + b_1 x + \dots b_n x^n$ with non-negative real coefficients such that $a(x) b(x) =: c(x)$, with $c(x)$ a 0--1 polynomial.

We use two different algorithms for studying this problem, and substantial computational verification. 

We first state a surprisingly useful, although simple result.

\begin{theorem}
\label{thm:reverse}
The polynomial $a(x) = a_0 + a_1 x + \dots + a_k x^k$ is a factor of an unfair polynomial
   if and only if $a^*(x) = a_k + a_{k-1} x + \dots a_0 x^k$ is a factor of an unfair polynomial.
\end{theorem}

In \cite{Ghidelli} the polynomial $1+ t x^2 + x^5$ was considered, where $0 < t < 1$.
Some of the techniques utilized in \cite{Ghidelli} are also used in this paper, albeit in a more 
    automated manner, (see Section \ref{sec:zos}).

For polynomials up to degree 5, all polynomials can be quickly shown to not be 
    factors of an unfair polynomial, with the exception of 
    $1+t x^2 + x^5$ and it's reciprocal.
Here, $0 < t < 1$.
Using techniques similar to those in Section \ref{sec:zos} it was shown in \cite{Ghidelli}
    that $b_n = 1 - t b_{n-3} - b_{n-5}$ for $5 \leq n \leq 10000$.
Further, the values $b_0, b_1, \dots, b_5$ are explicitly given in terms of 
    $t$.
For $0.005 \leq t < 1$, it was shown, based upon this linear recurrence, that 
    there exists an $n \leq 10000$ such that $b_n < 0$.
Hence, if $0.005 \leq t < 1$ then $a(x)$ is not a factor of an unfair polynomial.
If instead $0 < t < 0.005$, then an analysis on the location 
    of the roots of $a(x)$ was used to show that $b(x)$ must eventually have a 
    negative coefficient.

For the first step, we use a 
    simply trinary logic to show that most polynomials of degree at most 15 are not
    factors of an unfair polynomial.
This step was also used in \cite{Ghidelli}, albeit in a less automated manner.  This is discussed in Section \ref{sec:zos}.

In some cases the trinary logic is insufficient, but a recursive case 
    analysis can be utilized in addition to the trinary logic to derive a contradiction.
This is done in Section \ref{sec:recursion}.
    
A more sophisticated and computationally expense technique is given in Section \ref{sec:GB},
    utilizing Groebner basis and Quantifier Elimination.
It is interesting to note that we only needed to go up to $n=100$ to show that 
    $1 + t x^3 + x^5$ is not a factor of an unfair polynomial.
This is in contrast to $n = 10000$ which was used for $1 + t x^2 + x^5$ in \cite{Ghidelli}.

Section \ref{sec:comp} discusses the numerical results for all possible factors of degree less than or equal to 15.
In Section \ref{sec:alpha} we consider a relaxation of the definition of unfair polynomials for which solutions
    do exist.
Lastly, in Section \ref{sec:conc} makes some final remarks.

\section{Trinary logic}
\label{sec:zos}

In this section we will show how one can derive a contradiction using trinary logic.  A coefficient of $a(x)$ is one of three things.
It is either $0$, or $1$ or something strictly between $0$ and $1$.
If it is strictly between $0$ and $1$ we will denote it by $*$.
This logic was used in \cite{Ghidelli} to show most cases up to degree $5$ were not factors of unfair 0--1 polynomials, and to give the necessary structure for the last remaining degree $5$ cases.

A coefficient of $b(x)$ is one of four things. 
It may be $0$, $1$ or $*$ as before. 
In addition, it may simply be unknown as we do not have enough information to solve for it.

It is often possible to determine unknown coefficients of $b(x)$ to be one of 
    $0$, $1$ or $*$.
It is also often possible to determine a contradiction based upon known information.

Consider the product 
\begin{align*}
    &(a_0 + a_1 x + \dots a_k x^k) (b_0 + b_1 x + \dots b_n x^n) \\
    &\hspace{2cm}
= a_0 b_0 + (a_0 b_1 + b_0 a_1) x + \dots + \left(\sum_i a_i b_{j-i}\right) x^j + \dots .
\end{align*}
For convenience, we will denote $c_{i,j} = a_i b_j$.
This allows us to rewrite this product as
\begin{align}
    &  (a_0 + a_1 x + \dots a_k x^k) (b_0 + b_1 x + \dots + b_n x^n)  \nonumber \\ 
    & \hspace{2cm} =c_{0,0} + (c_{0,1} + c_{1,0}) x + \dots + \sum_i c_{i,j-i} x^j + \dots. \label{eq:c} 
\end{align}

By construction we have that $\sum_i c_{i, j-i}$ is either $0$ or $1$.
We also have that $a_0 = b_0 = 1$.
Below we give the trinary multiplication table:
\[
\begin{array}{l|lll}
\times  & 0 & 1 & * \\ \hline
 0 & 0 & 0 & 0 \\ 
 1 & 0 & 1 & * \\
 * & 0 & * & * 
 \end{array}
 \]
 If $a_i$ and $b_j$ are known, then we may use this information to 
    determine $c_{i,j}$.
If $c_{i,j}$ and $a_i$ are known, it is sometimes possible to use this 
    information to determine $b_j$.

From this, we can construct a table
\[
\begin{array}{l|llll}
 & b_0 & b_1 & b_2 & \dots \\ \hline
a_0 & c_{0,0} & c_{0,1} & c_{0,2} & \dots  \\
a_{1} & c_{1,0} & c_{1,1} & c_{1,2} & \dots \\
\vdots & \vdots & \vdots & \vdots &  \\
a_{k-1} & c_{k-1,0} & c_{k-1, 1} & c_{k-1, 2} & \dots \\
a_k & c_{k, 0} & c_{k, 1} & c_{k, 2} & \dots 
\end{array}
\]
By Equation \eqref{eq:c}, we observe that we can derive 
    the coefficients for $c(x)$ from this table by summing along the diagonal.
We note that this sum must always equal $0$ or $1$.

Initially, many of the values $b_j$ and $c_{i,j}$ are unknown.  
It is often possible to determine what these
    values must be, based on known information.

As $\sum_i c_{i,j-i}$ is either $0$ or $1$ we see that if $c_{i,j-i} = 1$ for some $i$ then 
    $c_{k, j-k} = 0$ for all $k \neq i$.
Additionally, if we know that $c_{i,j-i} = *$, and $c_{k, j-k}$ is unknown and lastly that $c_{\ell, j-\ell} = 0$ for all $\ell \neq i,k$ then we known that 
    $c_{k, j-k} = *$.
     
\begin{example}
Consider $a(x) = 1 + * x + x^3$.
We known that $b_0 = 1$.
This gives us the table
\[ \begin{array}{l|llll}
  & 1 & & & \\ \hline
1 & 1 \\ 
* & * \\
0 & 0 \\
1 & 1 \\
\end{array}
\]
By observing that $c_{1,0} + c_{0,1}  = * + c_{0,1}$ we get that $c_{0,1}  = *$.
By observing that $c_{3,0} + c_{2,1} + c_{1,2} + c_{0,3}  = 1 + c_{2,1} + c_{1,2} + c_{0,3}$ we
    get that $c_{2,1} = c_{1,2} = c_{0,3}  = 0$.

We can now update the table to get
\[
\begin{array}{l|llll}
  & 1 &   &   & \\ \hline
1 & 1 & * &   & 0\\ 
* & * &   & 0\\
0 & 0 & 0\\
1 & 1 \\
\end{array}
\]
By observing that $c_{0,1} = *$ and $a_0 = 1$ we get that $b_1 = *$.
By observing that $c_{1,2} = 0$ and $a_1 = *$ we get that $b_2 = 0$.
By observing that $c_{0,3} = 0$ and $a_0 = 1$ we get that $b_3 = 0$.

We can now update the table to get
\[
\begin{array}{l|llll}
  & 1 & * &  0 & 0\\ \hline
1 & 1 & * &   & 0\\ 
* & * &   & 0\\
0 & 0 & 0\\
1 & 1 \\
\end{array}
\]
Updating the table to fill in known multiplications gives
\[
\begin{array}{l|llll}
  & 1 & * & 0 & 0 \\ \hline
1 & 1 & * & 0 & 0\\ 
* & * & * & 0 & 0\\
0 & 0 & 0 & 0 & 0\\
1 & 1 & * & 0 & 0\\
\end{array}
\]
At this point we have a contradiction.
We have that $c_{2,0} + c_{1,1} + c_{0,2}  = 0 + * + 0 = *$ which is a number strictly between $0$ and $1$.
\end{example}

Such logic, when automated, and combined with Theorem \ref{thm:reverse}, can show a large number of $a(x)$ cannot be factors of an unfair polynomial.
See Section \ref{sec:comp}.

Using these techniques (combined with Theorem \ref{thm:reverse}), we can eliminate all cases for 
    $k = 2, 3, 4, 5$ and $6$ with the exception of
    \[ 
    1 + * x^2 + x^5,\hspace{1cm} 1 + * x + * x^2 + x^5, \hspace{1cm} 1 + * x + * x^2 + x^6
    \] 
and their reciprocals.

\section{Recursion}
\label{sec:recursion}

It occasionally happens that when we fill in the table using the techniques from Section \ref{sec:zos} that we do not arrive at a contradiction, and at the same time, we cannot derive any further information for the coefficients of $b(x)$.
As a coefficient of $b(x)$ is either $0$, $1$ or $*$, we can recursively check the three cases separately.

\begin{example}
Consider $a(x) = 1 + * x + * x^2 + x^5$.

Using the techniques from Section \ref{sec:zos}, we can construct the table (up to degree 6) to get
\[
\begin{array}{l|lllllll}
  & 1 & * &   & 0 & 0 & 0 & * \\ \hline 
1 & 1 & * &   & 0 & 0 & 0 & * \\
* & * & * & * & 0 & 0 & 0 & * \\
* & * & * &   & 0 & 0 & 0 & * \\
0 & 0 & 0 & 0 & 0 & 0 & 0 & 0 \\
0 & 0 & 0 & 0 & 0 & 0 & 0 & 0 \\
1 & 1 & * &   & 0 & 0 & 0 & * 
\end{array}
\]
At this point, no further information can be determined, and we have not reached a contradiction. We now recursively check $b_2$ as being either $0$, $1$ or $*$.

When $b_2 = 1$ we get the table
\[
\begin{array}{l|llllllll}
  & 1 & * & 1 & 0 & 0 & 0 & * \\ \hline 
1 & 1 & * & 1 & 0 & 0 & 0 & * \\
* & * & * & * & 0 & 0 & 0 & * \\
* & * & * & * & 0 & 0 & 0 & * \\
0 & 0 & 0 & 0 & 0 & 0 & 0 & 0 \\
0 & 0 & 0 & 0 & 0 & 0 & 0 & 0 \\
1 & 1 & * & * & 0 & 0 & 0 & * 
\end{array}
\]
which gives a contradiction as $c_{2,0} + c_{1,1} + c_{0,2}  = * + * + 1 > 1$.

When $b_2 = 0$ we get the table
\[
\begin{array}{l|llllllll}
  & 1 & * & 0 & 0 & 0 & 0 & * \\ \hline 
1 & 1 & * & 0 & 0 & 0 & 0 & * \\
* & * & * & * & 0 & 0 & 0 & * \\
* & * & * & 0 & 0 & 0 & 0 & * \\
0 & 0 & 0 & 0 & 0 & 0 & 0 & 0 \\
0 & 0 & 0 & 0 & 0 & 0 & 0 & 0 \\
1 & 1 & * & 0 & 0 & 0 & 0 & * 
\end{array}
\]
which gives a contradiction as $c_{1,2} = *$ and $b_2 = 0$.

Lastly, when $b_2 = *$ we get the table
\[
\begin{array}{l|llllllll}
  & 1 & * & * & 0 & 0 & 0 & * \\ \hline 
1 & 1 & * & * & 0 & 0 & 0 & * \\
* & * & * & * & 0 & 0 & 0 & * \\
* & * & * & * & 0 & 0 & 0 & * \\
0 & 0 & 0 & 0 & 0 & 0 & 0 & 0 \\
0 & 0 & 0 & 0 & 0 & 0 & 0 & 0 \\
1 & 1 & * & * & 0 & 0 & 0 & * 
\end{array}
\]
which gives a contradiction as $c_{4,0} + c_{3,1} + c_{2,2} + c_{1,3} + c_{0,4}  = *$ which is a number strictly between $0$ and $1$.
\end{example}

In some cases we have to use recursion multiple times on a candidate polynomial.

Using these recursive techniques (combined with Theorem \ref{thm:reverse}), we can eliminate all cases for 
    $k = 2, 3, 4, 5$ and $6$ with the exception of 
    $1 + * x^2 + x^5$ and it's reciprocal.

\section{Symbolic techniques}
\label{sec:GB}

The logic in Section \ref{sec:zos} did not use detailed information about non-integer values.
The only information we used was if a value was $0$, $1$ or strictly between $0$ and $1$.
It is possible to use more precise information about how unknown values interact.

In the previous section, we used a simplified form of multiplication for the values $0$, $1$ and $*$.  
Instead here, we use more precise information about the multiplication.  We note that $a_i b_j = c_{i,j}$.  
Hence, if information is known about these terms, then we have the identity $c_{i,j} - a_i b_j = 0$.  
In the previous section, we used the fact that $\sum_i c_{i, j-i}$ is either $0$ or $1$.  
As before, if we know that one of these terms is identically $1$, then all other terms must be zero.  
If all values of the diagonal are determined (i.e., not ``unknown''), then we have the additional polynomial identity $\left(\sum_i c_{i, j-i} \right)(1-\sum c_{i, j-i}) = 0$.  
We will call the set of all known identities at any step of the calculation the {\em basis of identities}, and we will denote this by $\mathcal{I}$.  
We can use quantifier elimination to determine if there is a solution to this basis of identities where all variables in the basis are strictly between $0$ and $1$.  
If there is not a solution, then we have derived a contradiction for this stage of the calculation. 
If instead, there is a solution, we continue to check more terms to (hopefully) derive a contradiction, or find a counter-example to 
    the unfair 0--1 polynomial conjecture.

This process uses recursion, although most branches of the recursion quickly lead to a contraction.

When adding new identities to $\mathcal{I}$, it is important to do this taking into account all known identities already in $\mathcal{I}$.  We can do this by representing $\mathcal{I}$ by its Groebner basis, and reducing all entries in the table with respect to this Groebner basis.

It is often the case that the variety represented by the basis of identities is actually the union of two or more sub-varieties.  (This is easy to check by looking at the Groebner basis).  If this is the case, we typically recurse on these sub-varieites, as this improves the performance of the calculations and the algorithm.

\begin{example}
Consider as an example the polynomial $1 + s x + t x^4 + x^7$.

Our initial table looks like
\[
\begin{array}{l|ll}
  & 1 & \\
\hline
1 & 1 & \\
s & s & \\
0 & 0 & \\
t & t & \\
0 & 0 & \\
0 & 0 & \\
0 & 0 & \\
1 & 1 &
\end{array}
\]

Filling in the easy to determine information gives
\[
\begin{array}{l|llllllll}
  & 1 & -  &-   & -  & 0 & -  & 0 & 0 \\
\hline
1 & 1 & -  & -  & -  & 0 & -  & 0 & 0  \\
s & s & -  & -  & -  & 0 & -  & 0 & 0  \\
0 & 0 & 0 & 0 & 0 & 0 & 0 & 0 & 0  \\
t & t & -  & -  & -  & 0 & -  & 0 & 0  \\
0 & 0 & 0 & 0 & 0 & 0 & 0 & 0 & 0 \\
0 & 0 & 0 & 0 & 0 & 0 & 0 & 0 & 0  \\
0 & 0 & 0 & 0 & 0 & 0 & 0 & 0 & 0 \\
1 & 1 &-   & -  & -  & 0 &  - & 0 & 0 
\end{array}
\]
At this point, we recurse.
Either $b_1 = 0$, $b_1 = 1$ or $0 < b_1 < 1$.

Both $b_1 = 0$ and $b_1 = 1$ quickly lead to a contradiction.  
So we may assume that $0 < b_1 < 1$.
This gives
\[
\begin{array}{l|lllllllll}
  & 1 & b_1 &  - & -  & 0 & -  & 0 & 0 & \\
\hline
1 & 1 & b_1 &-   &  - & 0 & -  & 0 & 0 & \\
s & s & s b_1 & -  & -  & 0 & -  & 0 & 0 & \\
0 & 0 & 0 & 0  & 0  & 0 &  0 & 0 & 0 & \\
t & t & t b_1 & -  &-   & 0 &  - & 0 & 0 & \\
0 & 0 & 0 & 0  & 0  & 0 &  0 & 0 & 0 & \\
0 & 0 & 0 & 0  & 0  & 0 &  0 & 0 & 0 & \\
0 & 0 & 0 & 0  & 0  & 0 & 0 & 0 & 0 & \\
1 & 1 & b_1 &-   &-   & 0 & -  & 0 & 0 &
\end{array}
\]
We also note that $s+ b_1 = 0$ or $s + b_1 = 1$.
As $s, b_1 > 0$ we easily see that the first cannot occur.
Hence we see that $s+b_1 = 1$.

We now update $\mathcal{I}$ to be \[ \mathcal{I} = \langle s + b_1 -1 \rangle.\]
We also have set of known inequalities.
\[ 0 < s,\ t,\ b_1 < 1.\]
A quick check with quantifier elimination ensures that there is a possible solution to 
\begin{align*}
    &\exists s,\ t,\ b_1:\\
    & \hspace{1cm} 0 < s < 1,\ 0 < t < 1,\ 0 < b_1 < 1, \\
    & \hspace{1cm} s+b_1 -1 = 0.
\end{align*}

We next recurse on $b_2$.  This may be $0$, $1$ or strictly between $0$ and $1$.
We quickly derive contradictions if $b_2 = 0$ or $b_2 = 1$.  
Hence we may assume that $0 < b_2 < 1$.
We perform all calculations modulo the basis of identities $\mathcal{I} = \langle s + b_1 - 1\rangle$
    (which is easy to do via Groebner basis).
This gives us the table
\[ 
\begin{array}{l|lllllllll}
  & 1 & -s+1 & b_2 & -  & 0 &  - & 0 & 0 & \\ 
\hline 
1 & 1 & -s+1 & b_2 & -  & 0 & -  & 0 & 0 & \\ 
s & s & -s^2+s & s b_2 & -  & 0 & -  & 0 & 0 & \\ 
0 & 0 & 0 & 0 & 0  & 0 & 0  & 0 & 0 & \\ 
t & t & -s t+t & t b_2 & -  & 0 & -  & 0 & 0 & \\ 
0 & 0 & 0 & 0 & 0  & 0 &0   & 0 & 0 & \\ 
0 & 0 & 0 & 0 & 0  & 0 &0   & 0 & 0 & \\ 
0 & 0 & 0 & 0 & 0  & 0 &0   & 0 & 0 & \\ 
1 & 1 & -s+1 & b_2 & -  & 0 &  - & 0 & 0 & 
\end{array}
\]

By considering the sum of the diagonal $c_{2,0} + c_{1,1} + c_{0,2}$, we see that $b_2 - s^2 + s$ is either $0$ or $1$.
As $b_2$ is strictly positive and $-s^2 + s$ is between $0$ and $1$, we see that this
    sum must equal $1$.
Hence we may add $b_2 - s^2 +s -1$ to our basis of identities.
This gives us 
\[ \mathcal{I} = \langle b_1 + s -1,\ b_2 -s^2 + s - 1 \rangle. \]
We then check via quantifier elimination to determine that there exists a solution to
\begin{align*}
    &\exists s,\  t,\  b_1,\  b_2: \\
    & \hspace{1cm} 0 < s < 1,\ 0 < t < 1,\ 0 < b_1 < 1,\ 0 < b_2 < 1,\\
    & \hspace{1cm} s+b_1 -1 = 0,\ b_2 - s^2 + s -1 = 0.
\end{align*}

We next recurse on $b_3$.  
As before we quickly derive a contradiction if $b_3 = 0$ or $b_3 = 1$.
When we assume $0 < b_3 < 1$, we can conclude that $s^3 - s^2 + s + t + b_3 -1 = 0$.
Expanding this table out, (reducing modulo $\mathcal{I}$), we get
\[ 
\begin{array}{l|lllllllll}
  & 1 & -s+1 & s^2-s+1 & -s^3+s^2-s-t+1 & 0 & -  & 0 & 0 & \\ 
\hline 
1 & 1 & -s+1 & s^2-s+1 & -s^3+s^2-s-t+1 & 0 & -  & 0 & 0 & \\ 
s & s & -s^2+s & s^3-s^2+s & -s^4+s^3-s^2-s t+s & 0 & -  & 0 & 0 & \\ 
0 & 0 & 0 & 0 & 0 & 0 & 0 & 0 & 0 & \\ 
t & t & -s t+t & s^2 t-s*t+t & -s^3 t+s^2 t-s t-t^2+t & 0 & -  & 0 & 0 & \\ 
0 & 0 & 0 & 0 & 0 & 0 & 0  & 0 & 0 & \\ 
0 & 0 & 0 & 0 & 0 & 0 &0  & 0 & 0 & \\ 
0 & 0 & 0 & 0 & 0 & 0 & 0  & 0 & 0 & \\ 
1 & 1 & -s+1 & s^2-s+1 & -s^3+s^2-s-t+1 & 0 &-   & 0 & 0 & 
\end{array}
\]
When summing along the diagonal $\sum_{i} c_{i, 4-i}$ we get $-s^4+s^3-s^2-2 s t+s+t$ we must be either $0$ or $1$.
When this is added to the basis of identities, quantifier elimination shows that there
    are no solutions to either
\begin{align*}
    & \exists s,\ t,\ b_1,\ b_2,\ b_3: \\
    & \hspace{1cm} 0 < s < 1,\ 0 < t < 1,\ 0 < b_1 < 1,\ 0 < b_2 < 1,\\
    & \hspace{1cm} s+b_1 -1 = 0,\ b_2 - s^2 + s -1 = 0,\ s^3 - s^2 + s + t + b_3 -1 = 0,\\
    & \hspace{1cm} -s^4+s^3-s^2-2 s t+s+t = 0
\end{align*}
or 
\begin{align*}
    & \exists s,\ t,\ b_1,\ b_2,\ b_3: \\
    & \hspace{1cm} 0 < s < 1,\ 0 < t < 1,\ 0 < b_1 < 1,\ 0 < b_2< 1,\\
    & \hspace{1cm} s+b_1 -1 = 0,\ b_2 - s^2 + s -1 = 0,\ s^3 - s^2 + s + t + b_3 -1 = 0,\\
    & \hspace{1cm} -s^4+s^3-s^2-2 s t+s+t-1=0.
\end{align*}
As such, we conclude that $1 + s x + t x^3 + x^7$ is not a potential factor of a unfair polynomial.
\end{example}

If upon some branch of this calculation we have a sequence of $k+1$ coefficients of $b(x)$ that are identically zero, and there exists a solution to $\mathcal{I}$ with all variables in $(0,1)$, then we have found a counter-example to the unfair 0--1 polynomial conjecture.  These tests were performed with the assumption that the degree of $b(x)$ was bounded by $200$.  Further, we used a four hour cap on the cpu time per test.  Under these restrictions, no test returned a counter-example to the unfair 0--1 polynomial conjecture.  Examples were found for a variation of the problem, as we will discuss in Section \ref{sec:alpha}.

\section{Numerical results}
\label{sec:comp}

In this section we present some of the results of our computational experiments. In Table \ref{tab:result} we indicate for each $k$, how many polynomials of degree $k$ there are
    with $a_0 = a_k = 1$ and at least one non-integer coefficient. 
    
The first two columns gives the number of polynomials that can be shown not to be factors of an
    unfair polynomial using the simple non-recursive trinary logic of Section \ref{sec:zos}, along
    with the time needed to perform these tests.
These techniques were very successful, showing a vast majority of the polynomials are not
    factors of unfair polynomials.
On average, these took 0.012 seconds per test.

The next two columns gives the number of polynomials than can be shown not to be factors of an
    unfair polynomial using simple recursive trinary logic of Section \ref{sec:recursion}, along
    with the time needed to perform these tests.
This was successful for 65\% of the remaining tests that were not resolved by non-recursive logic.
These were computationally more expensive, taking on average 1.4 seconds per tst.

The last two columns gives the number of polynomials that can be shown not to be factors of an
    unfair polynomial by using quantifier elimination as described in  Section \ref{sec:GB}.
This was successful for 75\% of the remaining tests.
This was by far the most computationally expensive test, taking on average 29 minutes per test.

\begin{landscape}
\begin{table}
\begin{tabular}{l|llllllll}
$k$ & Number   & Solved by & Time    & Solved         & Time  & Solved by & Time & Remaining\\
   &  of cases & Trinary   &  Req.   & by Recursive   & Req.  & Symbolic  & Req. & cases            \\
   &           & logic     &         & Trinary logic  &       & logic     &      &      \\
\hline
2 & 1 & 1 & .007 seconds & 0 & - seconds & 0 & - seconds & 0 \\ 
3 & 5 & 5 & .057 seconds & 0 & - seconds & 0 & - seconds & 0 \\ 
4 & 19 & 19 & .121 seconds & 0 & - seconds & 0 & - seconds & 0 \\ 
5 & 65 & 61 & .592 seconds & 2 & 3.0 seconds & 2 & 1.4 hours & 0 \\ 
6 & 211 & 209 & 1.6 seconds & 2 & .728 seconds & 0 & - seconds & 0 \\ 
7 & 665 & 643 & 6.5 seconds & 10 & 22.9 seconds & 6 & 8.7 hours & 6 \\ 
8 & 2059 & 2029 & 19.4 seconds & 20 & 24.9 seconds & 10 & 59.3 minutes & 0 \\ 
9 & 6305 & 6221 & 1.1 minutes & 48 & 1.6 minutes & 28 & 14.8 hours & 8 \\ 
10 & 19171 & 19025 & 3.5 minutes & 88 & 2.9 minutes & 48 & 23.5 hours & 10 \\ 
11 & 58025 & 57613 & 11.1 minutes & 206 & 10.2 minutes & 146 & 4.4 days & 60 \\ 
12 & 175099 & 174539 & 32.5 minutes & 390 & 10.9 minutes & 142 & 2.3 days & 28 \\ 
13 & 527345 & 525733 & 1.7 hours & 912 & 40.6 minutes & 522 & 14.3 days & 179 \\ 
14 & 1586131 & 1583587 & 5.4 hours & 1660 & 1.0 hours & 690 & 15.5 days & 194 \\ 
15 & 4766585 & 4760689 & 17.6 hours & 4100 & 2.3 hours & 1306 & 39.5 days & 490 \\ \hline 
Total  & 7141686 & 7130374 & 1.1 days & 7438 & 4.4 hours & 2900 & 78.0 days & 975 
\end{tabular}
\caption{Cases up to degree 15}
\label{tab:result}
\end{table}
\end{landscape}

All calculations were done using Maple 2023 \cite{Maple2023}.  
There were done on 4 machines, each with four Intel Xeon Gold 6230 20-core 2.1 GHz (Cascade Lake)
(768GB Memory).
Calculations were allowed to run for a maximum of 4 hours.

\section{$\alpha$-unfair polynomials}
\label{sec:alpha}

In the previous sections, we gave an algorithm to test if a potential factor was actually a factor of an unfair polynomial.
In this section, we will examine how much we would need to relax the definition of an unfair polynomial so that a factor would exist.
To this end, we give the 
\begin{definition}
    Let $c(x) = \sum c_i x^i$ with $c_i \in \{0,1\}$.
    We say that $c(x)$ is $\alpha$-unfair if there exists a factorization $c(x) = a(x) b(x)$ with $a(x) = \sum a_i x^i$, $b(x) = \sum b_i x^i$ and 
    \begin{itemize}
        \item $b_0 = a_0 = 1$,
        \item There exists an $i$ such that $b_i \not\in \{0,1\}$,
        \item $-\alpha \leq b_i \leq 1+\alpha$,
        \item $-\alpha \leq a_i \leq 1+\alpha$.
    \end{itemize}
\end{definition}

We easily see that a $0$-unfair polynomial is an unfair polynomial from Section \ref{sec:intro}.
By Table \ref{tab:result}, we see that there do not exist $0$-unfair polynomials with a factor of degree less than or equal to $6$.

The methods of Section \ref{sec:GB} can be modified to search for for examples of $\alpha$-unfair polynomials.  There are two key differences.  The first is, when testing if there is a solution via quantifier elimination, we test in the variables are in the range $[-\alpha, 1+\alpha]$ instead of $(0,1)$.  The second difference is that we can no longer assume that if $c_{i,j} =1$ then all other terms in the diagonal must be identically $0$.

\begin{example}
    Let $\alpha = 1$ and $a(x) = 1 + s x + x^2$.
We quickly get the  table
\[ 
\begin{array}{l|ll}
  & 1 & \\ 
\hline 
1 & 1 & \\ 
s & s & \\ 
1 & 1 & 
\end{array}
\]
We then recurse on $b_1$, testing if it is $0$, $1$ or a value in $[-\alpha,1+\alpha] \setminus {0,1} = [-1, 0) \cup (0,1) \cup (1, 2]$.
We derive contradictions if $b_1$ is $0$ or $1$, giving us
\[ 
\begin{array}{l|lll}
  & 1 & b_1 & \\ 
\hline 
1 & 1 & b_1 & \\ 
s & s & s b_1 & \\ 
1 & 1 & b_1 & 
\end{array}
\]
Here the diagonal $s + b_1$ must be either $0$ or $1$.  
In this case, we do not derive a contradiction to this sum is $0$.
Hence, we first test the case where $\mathcal{I} = \langle s + b_1 \rangle$ and $b_1 = -s$.
This gives
\[ 
\begin{array}{l|lll}
  & 1 & -s & \\ 
\hline 
1 & 1 & -s & \\ 
s & s & -s^2 & \\ 
1 & 1 & -s & 
\end{array}
\]
Continuing in this fashion, one of the branches gets to
\[ 
\begin{array}{l|llllllllllll}
  & 1 & -s & 1+s & -s & 1 & 0 & 0 & 0 & 0 & 0 &\\ 
\hline 
1 & 1 & -s & 1+s & -s & 1 & 0 & 0 & 0 & 0 & 0 &  \\ 
s & s & -1-s & 2*s+1 & -1-s & s & 0 & 0 & 0 & 0 & 0 &  \\ 
1 & 1 & -s & 1+s & -s & 1 & 0 & 0 & 0 & 0 & 0 &  & 
\end{array}
\]
with $\mathcal{I} = \langle s^2-s-1, s+b_1, -1-s+b_2, b_3+s \rangle$.
This does have a solution with all variables in the range $(-1, 2)$.
This gives us an example of a $1$-unfair polynomial.  
In particular, we have that
\begin{align*}
x^6 + x^4+x^3+x^2+1&  = (1+s x + x^2) (1 - s x + (1+x) x^2 - s x^3 + x^4)  \\
& \approx
    (1-0.618 x + x^2) (1 + 0.618 x + 0.381 x^2 + 0.618 x^3 + x^4) .
\end{align*}
where $s \approx -0.618$ is the root of $s^2-s-1$.

This in fact is an example of a $\frac{\sqrt{5}-1}{2}$-unfair polynomial.
\end{example}

We observe
\begin{lemma}
    If there exists an $\alpha$-unfair polynomial with a factor of degree $k$, then for all $1 \leq i$ and $0 \leq j \leq i-1$ there exists an $\alpha$-unfair polynomial with a factor of degree $ik+j$.
\end{lemma}

\begin{proof}
Let $a(x)$ and $ b(x)$ be factors of an $\alpha$-unfair polynomial, with
    $a(x)$ of degree $k$.
We see that $a(x) b(x)$ is a 0--1 polynoimal.
Define $A(x) = a(x^i)$ and $B(x) = b(x^i)$.
We see that $A(x) B(x)$ is a 0--1 polynomial, and hence there exists a
    factor of an $\alpha$-unfair polynomial of degree $ki$.
Similarly if $j \neq 0$ we can take $A(x) = a(x^i) (x^j+1)$ and $B(x) = b(x^i)$ to get a factor of an $\alpha$-unfair polynomial of degree $ki+j$.
\end{proof}

We used this algorithm to computationally explore how small we could
    have $\alpha$ such that there exists an $\alpha$-unfair polynomial.
To do this, we started with a reasonably large $\alpha$, say $\alpha =1$, and ran the algorithm.
If we found a solution, we would then update $\alpha$ to be slightly
     smaller than the $\alpha$ generated by this solution and repeat.
Based upon these experiments, and verification of an observed pattern, we make the 
\begin{conj}
\label{conj:Crazy3}
For all $\epsilon > 0$ there exists an $\epsilon$-unfair polynomial.
\end{conj}

\begin{proof}[Observations for Conjecture \ref{conj:Crazy3}]
Let $n \geq 1$.
Let
\[ C(x) = 1 + x^3 + x^{6n+2} + x^{6n+3} + x^{6n+4} + x^{6n+5} + 
x^{6n+6} + x^{12n+6} + x^{12n+9}.\]
Take $A(x,t) = x^3 + t x^2 + t x + 1$.
Write $C(x) = A(x,t) B(x,t) + R(x, t)$ for 
    $R(x, t)$ of degree $2$ with respect to $x$.
Write $R(x, t) = r_2(t) x^2 + r_1(t) x + r_0(t)$.
Write $g(t) = \gcd(r_2(t), r_1(t), r_0(t))$.
Computationally $g$ is non-trivial, and has degree $6n+4$.
Let $t_n$ to the smallest positive root of $g$.
Find the minimal $\alpha_n$ such that 
    the coefficients of $A(x, t_n)$ and $B(x, t_n)$ are in $[-\alpha_n, 1+\alpha_n]$.
Computationally $\alpha_n$ appears to be tending to $0$ as $n \to \infty$.
See Table \ref{tab:alpha}.
\end{proof}

\begin{table}
    \begin{tabular}{l|ll}
    $n$ & $t_n$ & $\alpha_n$ \\ \hline
1&  .037612293& .039963513\\ 
2&  .013407581& .014053745\\ 
3&  .006822026& .007081092\\ 
4&  .004122287& .004250608\\ 
5&  .002757984& .002830563\\ 
6&  .001974008& .002018957\\ 
7&  .001482396& .001512129\\ 
8&  .001153957& .001174632\\
9&  .000923716& .000938668\\
10& .000756095& .000767256 
\end{tabular}
\caption{Details of $\alpha$-unfair polynomials for Conjecture \ref{conj:Crazy3}}
\label{tab:alpha}
\end{table}

\begin{remark}
\label{rmk:alpha}
It appears that $\log(\alpha_n) \approx
 -2.925903281 - 1.871057363 \log(12 n + 9)$.
 This was based on the data for $1 \leq n \leq 40$.
See Figure \ref{fig:Fit}.
\end{remark}

\begin{figure}[ht]
    \centering
    \includegraphics[scale=0.5,angle=270]{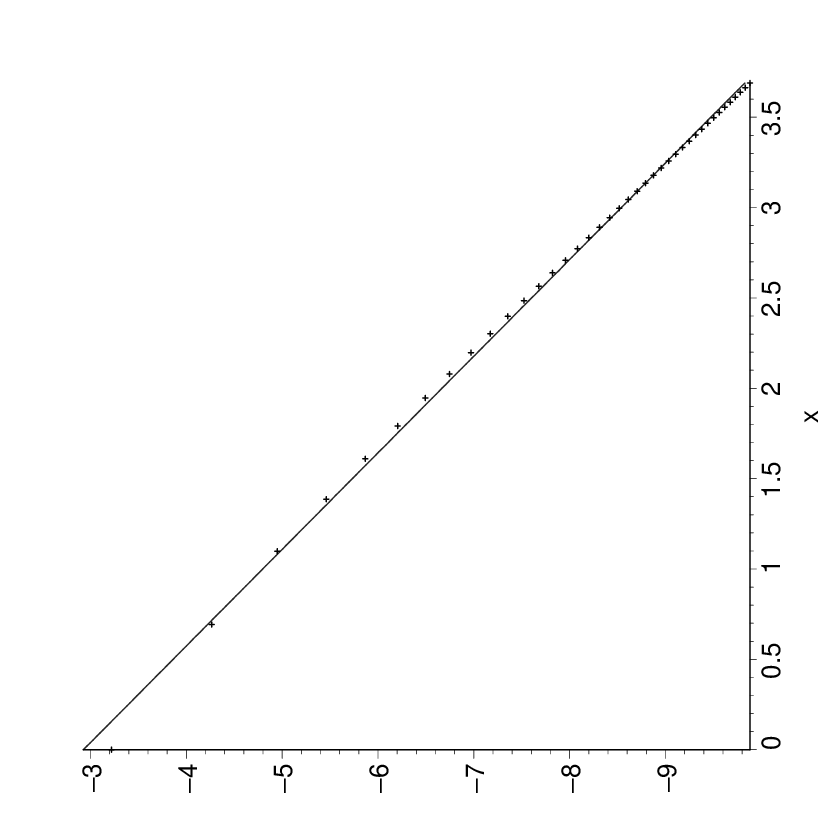}
    \caption{Plot $[\log(12 n + 9),\log(\alpha_n)]$ for $1 \leq n \leq 40$, along with line of best fit for Remark \ref{rmk:alpha}}
    \label{fig:Fit}
\end{figure}

\section{Final Remarks}
\label{sec:conc}

In this paper we computationally searched for unfair 0--1 polynomials with a given potential factor.  Although we were not able to find an example of such a polynomial, we were unable to prove that one does not exist.
In Section \ref{sec:comp} we have provided considerable computational evidence that such polynomials do not exist.
If it is true that a polynomial does not exist, then there appears to be a clear divide between $\alpha$-unfair 0--1 polynomials with $\alpha > 0$ and $0$-unfair 0--1 polynomial.

\section{Acknowledgments}

I would like to thank Juergen Gerhard for informing me 
    of the new Quantifier Elimination package in Maple 2023.

\bibliographystyle{plain}

\end{document}